\documentclass[a4paper,11pt]{article}
\usepackage{geometry}                
\usepackage[active]{srcltx}
\usepackage{hyperref}

\usepackage{graphicx}
\usepackage{amssymb}
\usepackage{epstopdf}
\usepackage{amsmath}
\usepackage{amsthm}
\usepackage{amssymb}
\usepackage{latexsym}
\usepackage{mathtools}
\usepackage{mathrsfs}
\usepackage{graphics}
\usepackage{latexsym}
\usepackage{psfrag}
\usepackage{import}
\usepackage{verbatim}
\usepackage{enumerate}
\usepackage{enumitem}
\usepackage{graphicx}
\usepackage{esint}
\usepackage[usenames]{color}

\newcommand{\dom}{{\textrm{Dom\,}}}
\newcommand{\R}{\mathbb{R}}

\newcommand{\NN}{\mathcal{N}}
\newcommand{\cH}{\mathcal{H}}

\newcommand{\supp}{\text{\rm supp}}

\newcommand{\ve}{\varepsilon}

\newcommand{\gra}{\textrm{graph}}

\newcommand{\T}{\mathcal{T}}

\newcommand{\RCD}{\mathsf{RCD}}

\newcommand{\CD}{\mathsf{CD}}

\newcommand{\Geo}{{\rm Geo}}
\newcommand{\MCP}{\mathsf{MCP}}

\newcommand{\OptGeo}{{\rm OptGeo}}

\newcommand{\set}[1]{\left\{#1\right\}}

\newcommand{\Real}{\mathbb{R}}

\newcommand{\eps}{\varepsilon}

\newcommand{\m}{\mathfrak{m}}
\newcommand{\q}{\mathfrak{q}}

\renewcommand{\P}{\mathbb P}

\renewcommand{\P}{\mathcal{P}}

\newcommand{\mm}{\mathfrak m}
\newcommand{\qq}{\mathfrak q}

\newcommand{\ee}{{\rm e}}

\newcommand{\sfd}{\mathsf d}

\newcommand{\Opt}{\mathrm{OptGeo}}

\theoremstyle{plain}
\newtheorem{lemma}{Lemma}[section]
\newtheorem{theorem}[lemma]{Theorem}

\newtheorem{proposition}[lemma]{Proposition}

\newtheorem*{theorem*}{Theorem}
\newtheorem*{maintheorem*}{Main Theorem}

\theoremstyle{definition}

\newtheorem{definition}[lemma]{Definition}
\newtheorem*{definition*}{Definition}
\newtheorem{remark}[lemma]{Remark}

\numberwithin{equation}{section}

\newcommand{\ppi}{{\mbox{\boldmath$\pi$}}}

\title{Displacement convexity of Entropy \\ 
and the distance cost Optimal Transportation}

\author{Fabio Cavalletti\thanks{F. Cavalletti: Mathematics Area, SISSA, Trieste (Italy), email: cavallet@sissa.it.} , Nicola Gigli\thanks{N. Gigli: Mathematics Area, SISSA, Trieste (Italy), 
email: ngigli@sissa.it} \ and Flavia Santarcangelo\thanks{F. Santarcangelo: Mathematics Area, SISSA, Trieste (Italy), email: fsantarc@sissa.it}}

\date{}     

\begin{document}
\maketitle

\begin{abstract}
During the last decade Optimal Transport had a relevant role in the study of geometry of 
singular spaces that culminated with the Lott-Sturm-Villani theory. 
The latter is built on the characterisation of Ricci curvature lower bounds 
in terms of displacement convexity 
of certain entropy functionals along $W_{2}$-geodesics. 
Substantial recent advancements in the theory 
(localization paradigm and local-to-global property)
have been obtained considering the different point of view of $L^{1}$-Optimal transport 
problems yielding a different curvature dimension $\CD^{1}(K,N)$ \cite{CMi} formulated 
in terms of one-dimensional curvature properties of integral curves of Lipschitz maps.
In this note we show that the two approaches produce the same curvature-dimension condition 
reconciling the two definitions. In particular we show that the $\CD^{1}(K,N)$ condition 
can be formulated in terms of displacement convexity along $W_{1}$-geodesics. 

\end{abstract}

\bibliographystyle{plain}

\section{Introduction}

The formulation of an appropriate version of Ricci curvature lower bounds 
valid for possibly singular spaces has been a central topic of research for several years. 
During the last decade Optimal Transport had a relevant role in the topic 
that culminated with the successful theory of Lott-Villani  
\cite{lottvillani:metric} and Sturm \cite{sturm:I,sturm:II} of metric measure spaces 
verifying a lower bound on the Ricci curvature in a synthetic sense.

The theory is formulated in terms of displacement convexity  of the
Renyi entropy. The latter is defined on the set of probability measures 
$S_N ( \cdot| \m):{\mathcal{P}_2(X,\sfd)} \to \mathbb{R}$ as follows
\[
S_N(\mu | \m):= -\int_{X} \rho^{-1/N}\,d\mu,
\]
where $\rho$ denotes the density of the absolutely continuous part 
of $\mu$ with respect to $\m$. In rough terms, a space will 
satisfy the $\CD(K,N)$ condition if the entropy evaluated along 
$W_{2}$-geodesics is more convex than the 
entropy evaluated along $W_{2}$-geodesics of the model space 
with constant curvature $K$ and dimension $N$ in an appropriate 
sense (see Definition \ref{D:sturmII}). 

The theory had a huge impact 
and a detailed discussion on its development would be beyond the scope of the note. 
For our purposes, we mention that substantial recent advancements in the theory 
(localization paradigm and local-to-global property)
have been obtained considering the different point of view of $L^{1}$-Optimal transport 
problems yielding a different curvature dimension $\CD^{1}(K,N)$ \cite{CMi} formulated 
in terms of one-dimensional curvature properties of integral curves of Lipschitz maps.

Motivated by the proof of the local-to-global property for the curvature-dimension condition, in \cite{CMi} 
has been shown that a metric measure space $(X,\sfd,\mm)$ verifies 
$\CD(K,N)$ if and only if it satisfies $\CD^{1}(K,N)$, 
provided $X$ is essentially non-branching (see Definition \ref{D:W2enb})
and the total space to have finite mass (i.e. $\mm(X) < \infty$).

Moreover it was recently addressed whether or not the $\CD$ condition really depends 
on the special exponent $p = 2$ used to check displacement convexity of entropy. 
While for smooth manifold it is clear that it does not (being equivalent to a lower bound on the Ricci tensor) the general case of metric measure spaces has been considered 
in the recent \cite{CCMcCAS} where complete equivalence will be proved.
It remained however unclear if the $\CD^{1}(K,N)$ condition 
could  be equivalently formulated in terms of displacement convexity of the Entropy functional along $W_{1}$-geodesics. 

In this note we show that this is the case and the two approaches produce the same curvature-dimension condition reconciling the two definitions. 
We report here the main result of the paper. 

\begin{theorem}\label{T:Mainintro}
Let $(X,\sfd,\m)$ be an essentially non-branching metric measure space
and further assume $\mm(X) = 1$. 
Then $(X,\sfd,\m)$ satisfies the $\CD^1(K,N)$ condition if and only if 
it satisfies the $\CD_{1}(K,N)$ condition.
\end{theorem}

The $\CD_{1}(K,N)$ condition is formulated, in analogy with the classical 
$\CD(K,N)$, as displacement convexity of entropy along $W_{1}$-geodesics;
its precise formulation is given in Definition \ref{D:CD1entropy}.


\section{Background material}
In this section we will recall some  basic notions  used throughout the paper.

A triple $(X,\sfd,\m)$  is called a \emph{metric measure space}  if $(X, \sfd)$ is a Polish space (i.e. a  complete and separable metric space) and $\m$ is  a positive
Radon measure over $X$. In what follows we will always deal with m.m.s. in which 
 $\m$ is a  probability measure, i.e. $\m(X)=1$; we will denote with $\mathcal{P}(X)$ the space of all Borel probability
measures over $X$.

A curve $\gamma\in{C([0,1],X)}$ is called a  \emph{constant speed geodesic} if 
\[
\sfd(\gamma_s, \gamma_t)= |s-t|d(\gamma_0,\gamma_1),\,\,\,\,\forall s,t\in{[0,1]}.
\]
From now on the set of  all constant speed geodesics will be denoted with 
$\Geo(X)$ while  $\ee_t:\Geo(X)\to X$ will denote the \emph{evaluation map}  defined by
$\ee_t(\gamma)=\gamma_t$.
Moreover we will call $(X,\sfd,\m)$ \emph{geodesic} if, for any choice of $x, y\in{X}$, there exists  $\gamma\in{\Geo(X)}$ with $\gamma_0=x$, $\gamma_1=y$.

As usual, for any  $p\geq 1$, $\P_p(X)$ will denote 
the space of probability measures with finite $p$-moment, i.e. 
\[
\mathcal{P}_p(X)= \{ \m\in{\mathcal{P}(X)}: \int_{X} \sfd^p(x,x_0)\,\m(dx) < +\infty, \text{for some } x_0\in{X}\},
\]
and with  $\mathcal{P}_p(X,\sfd,\m)$ its subspace of $\m$-absolutely continuous probability. 
The space $\mathcal{P}_p(X)$ will be  endowed with the $L^p$-Wasserstein distance $W_p$ defined by 
\begin{equation}\label{equ:Wp}
W_p(\mu_0,\mu_1)=\left(\inf_{\pi} \int_{X\times X}\sfd^p(x,y) \pi(dxdy) \right)^{1/p},
\end{equation}
where the infimum is taken in the class of all probability measures in $\mathcal{P}(X\times X)$ with first and second marginal given by $\mu_0$ and $\mu_1$ respectively.

It is a classical fact that if $(X,\sfd)$ is geodesic then $(\P_p(X), W_p)$ is geodesic too
and a curve $[0,1]\ni{t}  \mapsto \mu_t\in{\P_p(X)}$ is a geodesic
if and only if there exists $\nu\in{\mathcal{P}(\Geo(X))}$ such that 
$(\ee_0, \ee_1)_\sharp \nu$ realizes the minimum in \eqref{equ:Wp} and $\mu_t={\ee_t}_{\sharp}\nu$.
We will summarize these two properties saying that $\nu$ is an optimal dynamical plan
 $\nu \in{\OptGeo_{p}(\mu_0,\mu_1)}$.
Finally, $ A\subset \Geo(X)$ is called \emph{a set of non-branching geodesics} 
if  for any $\gamma^1, \gamma^2 \in{A}$
\[
\exists\, \bar{t}\in{(0,1)}:\gamma^1(s)=\gamma^2(s),\,\, \forall s\in{[0,\bar{t}]} \implies \gamma^1(t)=\gamma^2(t), \,\,\forall t\in{[0,1]}.
\]

Finally we recall the classical definition of essentially non-branching. 
This notion has been firstly introduced in \cite{RS2014} and considers only the case $p =2$.

\begin{definition}[Essentially non-branching]\label{D:W2enb}
Let $(X,\sfd, \mm)$ be a m.m.s.. We say that $(X,\sfd,\m)$  is $W_{2}$-essentially non-branching if 
for any $\mu_{0},\mu_{1} \in \mathcal{P}_2(X,\sfd,\mm)$ any element of $\Opt_{2}(\mu_{0},\mu_{1})$ 
is concentrated on a set of non-branching geodesics.
\end{definition}

\subsection{$L^1$-Optimal Transport}\label{Ss:1OTu}
To any $1$-Lipschitz function $u : X \to \R$ can be naturally associated  a $\sfd$-cyclically monotone set $\Gamma_u$ defined  in the following way:
$$
\Gamma_{u} : = \{ (x,y) \in X\times X : u(x) - u(y) = \sfd(x,y) \}.
$$
We define the \emph{transport relation} $R_u$ and the \emph{transport set} $\mathcal{T}_{u}$ in the following way:
\begin{equation}\label{E:R}
R_{u} := \Gamma_{u} \cup \Gamma^{-1}_{u} ~,~ \mathcal{T}_{u} := P_{1}(R_{u} \setminus \{ x = y \}) ,
\end{equation}
where $\{ x = y\}$ denotes the diagonal $\{ (x,y) \in X^{2} : x=y \}$, 
$P_{i}$ the projection onto the $i$-th component 
and  $\Gamma^{-1}_{u}= \{ (x,y) \in X \times X : (y,x) \in \Gamma_{u} \}$.

\noindent
Since $u$ is $1$-Lipschitz, $\Gamma_{u}, \Gamma^{-1}_{u}$ and $R_{u}$ are closed sets, and so are $\Gamma_u(x)$ and $R_u(x)$   (recall that $\Gamma_u(x) = \set{y \in X \; ;\; (x,y) \in \Gamma_u}$ 
and similarly for $R_u(x)$).
Consequently $\mathcal{T}_{u}$ is a projection of a Borel set and hence it is analytic; it follows that it is universally measurable, and in particular, $\mm$-measurable \cite{Srivastava} .

The transport ``flavor'' of  the previous definitions can be seen in the next property that is immediate to verify: for any $\gamma\in \Geo(X)$ such that 
$(\gamma_0,\gamma_1)\in \Gamma_{u}$, then
\[
(\gamma_s,\gamma_t)\in \Gamma_{u},\quad \forall \  0\leq s\leq t\leq 1.
\]

Finally, recall the definition of the sets of  \emph{forward and backward branching points} introduced in \cite{cava:MongeRCD}:
\begin{align*}
&A_{+,u}:=\{x\in{\T_{u}: \exists\,z,w\in{\Gamma_{u}(x)}, (z,w)\notin{R_{u}}}\},\\
&A_{-,u}:=\{x\in{\T_{u}: \exists\,z,w\in{\Gamma_{u}(x)^{-1}}, (z,w)\notin{R_{u}}}\}.
\end{align*}

Once branching points are removed, we obtain the \emph{non-branched transport set}  and the \emph{non-branched transport relation},
\begin{equation}\label{E:branch}
\T^{b}_{u}:=\T_{u}\setminus (A_{+,u}\cup A_{-,u}), \quad
R_{u}^b:=R_{u}\cap (\T_{u}^b\times \T_{u}^b);
\end{equation}
the following was obtained in \cite{cava:MongeRCD} and highlights the motivation
to remove branching points.

\begin{proposition}
The set of transport rays $R_{u}^b\subset X\times X$ is an equivalence relation on the set $\T_{u}^b$.
\end{proposition}

\noindent
Noticing that once we fix $x\in \T_{u}^b$, for any choice of  
$z,w\in R_{u}(x)$, there exists $\gamma\in  \Geo(X)$ such that 
\[
\{x,z,w\} \subset  \{\gamma_s: s \in{[0,1]}\},
\]
it is not hard to deduce that each equivalence class is a geodesic. 

\smallskip
The next step is to use this partition of the transport set made of equivalence classes 
to obtain a corresponding decomposition of the ambient measure $\mm$ restricted 
to $\T^b_{u}$.
Disintegration Theorem (for an account on it see \cite{biacar:cmono}) will be
the appropriate technical tool to use. The first step is to obtain 
an $\mm$-measurable quotient map 
$f$ for the equivalence relation $R^{b}_{u}$ over $\T_{u}^{b}$ whose construction  
is by now a classical procedure. It is worth stressing that the quotient set 
will be identified with with a subset of $\T_{u}^{b}$ containing a point 
for each equivalence class, i.e. for each geodesic forming $\T_{u}^{b}$.
In particular, there will be an $\mm$-measurable quotient 
set $Q \subset \T_{u}^{b}$, image of $f$.
The Disintegration Theorem (for an account on it see \cite{biacar:cmono}) then implies the following disintegration formula:
\begin{equation}\label{E:disint}
\m \llcorner_{\T^b_{u}}= \int_{Q} \m_{\alpha} \mathfrak{q}(d\alpha),
\end{equation}
where $\qq = f_{\sharp} \m \llcorner_{\T_{u}^b}$, and for 
$\qq$-a.e. $\alpha \in Q$
we have $\mm_{\alpha} \in \mathcal{P}(X)$, 
$\mm_{\alpha}(X \setminus X_{\alpha}) = 0$, where we have used the notation $X_{\alpha}$ to denote the equivalence class of the element $\alpha \in Q$ 
(indeed $X_{\alpha} = R(\alpha)$).
 In \cite{cava:MongeRCD}, it was proved  that under $\RCD(K,N)$ condition 
the measure of the sets of branching points is zero. As already observed several times in the literature, the proof only requires existence and uniqueness of optimal maps for $p = 2$.

\begin{theorem}\label{teo:branch}
Let $(X,\sfd, \mm)$ be a m.m.s. such that for any 
$\mu_0,\mu_1\in{\mathcal{P}_{2}(X)}$ with $\mu_0 \ll \m$ any $W_2$-optimal transference plan 
is concentrated on the graph of a function. Then for every 1-Lipschitz function $u:X\to\R$ we have
\[
\m(A_{+,u})=\m(A_{-,u})=0.
\]
\end{theorem}

It is worth here recalling that if $(X,\sfd, \mm)$ verifies $\MCP(K,N)$ and is essentially non-branching, 
then \cite{CavallettiMondino17c} implies that $(X,\sfd, \mm)$ verifies the assumptions of Theorem \ref{teo:branch} 
implying $\m(A_{+,u})=\m(A_{-,u})=0$, for any $u : X \to \R$ 1-Lipschitz function.

\subsection{Curvature-Dimension conditions}
We conclude this section by quickly recalling the main definitions of synthetic Ricci curvature lower bounds 
relevant to this note. 
We start with the first one that has been given by Lott-Villani \cite{lottvillani:metric} and Sturm \cite{sturm:I},\cite{sturm:II}.


Given a  metric measure space $(X,\sfd,\m)$ and $N\in{\mathbb{R}},N\geq 1$, we define the \emph{Renyi entropy functional} $S_N ( \cdot| \m):{\mathcal{P}_2(X,\sfd)} \to \mathbb{R}$ as follows
\[
S_N(\mu | \m):= -\int_{X} \rho^{-1/N}\,d\mu,
\]
where $\rho$ denotes the density of the absolutely continuous part of $\mu$ with respect to $\m$.
We also recall the definition of distortion coefficients. 
For every $K, N\in{\mathbb{R}}$  with $N\geq 1$, we set
\[
D_{K,\NN} := \begin{cases}  \frac{\pi}{\sqrt{K/\NN}}  & K > 0 \;,\; \NN < \infty, \\ +\infty & \text{otherwise.} \end{cases}
\]
Given $t \in [0,1]$ and $0 < \theta < D_{K,\NN}$, 
the  \emph{distortion coefficients}  $\sigma^{(t)}_{K,N}(\theta)$ are defined by
\begin{equation*}
\sigma^{(t)}_{K,N}(\theta):= 
\begin{cases}
\infty& {\rm if}\ K\theta^2\geq N\pi^2, \crcr
\frac{\sin(t\theta\sqrt{K/N})}{\sin(\theta\sqrt{K/N})} & {\rm if}\ 0<K\theta^2<N\pi^2, \crcr
t & {\rm if}\, K\theta^2<0\, {\rm and}\, N=0,\, {\rm or\, if}\,K\theta^2=0, \crcr
\frac{\sinh(t\theta\sqrt{-K/N})}{\sinh(\theta\sqrt{-K/N})} & {\rm if}\ K\theta^2\leq0\, {\rm and}\, N>0. \crcr
\end{cases}
\end{equation*}
Finally, given $K \in \Real$,  $N \in (1,\infty]$  and $(t,\theta)\in{[0,1]\times\R_+}$, 
$\tau_{K,N}^{(t)}(\theta) := t^{\frac{1}{N}} \sigma_{K,N-1}^{(t)}(\theta)^{1 - \frac{1}{N}}$.
When $N=1$, set $\tau^{(t)}_{K,1}(\theta) = t$ if $K \leq 0$ and $\tau^{(t)}_{K,1}(\theta) = +\infty$ if $K > 0$.

\begin{definition}\label{D:sturmII}[\cite{sturm:II}]
Given two numbers $K,N\in{\mathbb{R}}$ with $N\geq1$ we say that a metric measure space $(X,\sfd, \m)$ satisfies the \emph{curvature-dimension condition} $\CD(K,N)$ if and only if for each pair of $\mu_0,\mu_1\in{\mathcal{P}_2(X,\sfd,\m)}$ there exist an optimal coupling $\pi$ of $\mu_0=\rho_0\m$ and $\mu_1=\rho_1\m$ and a $W_2$-geodesic $\{\mu_t\}$ interpolating the two such that
\begin{equation}
\label{equ:cdd}
S_{N'}(\mu_t | \m) \leq -\int_{X\times X} \bigl[\tau_{K,N'}^{(1-t)}(\sfd(x,y)) \rho_0^{-1/N'}(x)+\tau_{K,N'}^{(t)}(\sfd(x,y)) \rho_1^{-1/N'}(y) \bigr] d\pi(x,y)
\end{equation}
for all $t\in{[0,1]}$ and all $N'\geq N$.
\end{definition}

One can also prescribe the convexity inequality \eqref{equ:cdd} to hold along a $W_p$-geodesic, getting to the more general definition of  $\CD_{p}(K,N)$.
In this note we will deal with the case $p=1$, that, due to the lack of 
strict convexity of the exponent, needs a more refined definition.

\begin{definition}\label{D:CD1entropy}
Given two numbers $K,N\in{\mathbb{R}}$ with $N\geq1$ we say that 
a metric measure space $(X,\sfd, \m)$ satisfies the  $\CD_{1}(K,N)$ 
if and only if for each pair of $\mu_0,\mu_1\in{\mathcal{P}_1(X,\sfd,\m)}$ there exists a Borel probability measure $\ppi\in\mathcal P(C([0,1],X))$ concentrated on constant speed geodesics, such that $\int\sfd(\gamma_0,\gamma_1)\, d\ppi(\gamma)=W_1(\mu_0,\mu_1)$ for which the inequality
\begin{equation}
\label{equ:cd}
S_{N'}(\mu_t | \m) \leq -\int_{X\times X} \bigl[\tau_{K,N'}^{(1-t)}(\sfd(\gamma_0,\gamma_1)) \rho_0^{-1/N'}(\gamma_0)+\tau_{K,N'}^{(t)}(\sfd(\gamma_0,\gamma_1)) \rho_1^{-1/N'}(\gamma_1) \bigr] d\ppi(\gamma)
\end{equation}
holds for all $t\in{[0,1]}$ and all $N'\geq N$, 
where $\mu_t:=({\rm e}_t)_*\ppi$ and $\mu_{t} \ll \mm$ and 
$(\ee_{i})_{\sharp} \ppi = \mu_{i}$ for $i = 0,1$.
\end{definition}
\begin{remark}
Notice that since we are dealing with the 1-transportation distance, there are dynamic transport plans which are not concentrated on constant speed geodesics. Insisting on this property in the definition above seems the natural choice to make in connection with the analogous definitions for $p>1$, see e.g.\ Lemma \ref{le:32}.
\end{remark}

We now recall the definition of the  
$\CD^{1}(K,N)$ condition introduced in \cite{CMi} and 
based on another principle: the localization of Ricci curvature lower bounds along 
integral curves associated to 1-Lipschitz function.

\begin{definition}\label{D:defCD1}
($\CD^1(K,N)$ when $\supp(\m)=X$) Let $(X,\sfd,\m)$ be a metric measure space such that $\supp(\m)=X$. Let us consider $K,N\in{\mathbb{R}}$, $N>1$ and let $u:(X,\sfd)\to \mathbb{R}$ be a 1-Lipschitz function. We say that $(X,\sfd,\m)$ satisfies the $\CD^1_u(K,N)$ condition if there exists a family $\{X_{\alpha}\}_{\alpha\in{Q}}\subset X$ such that :

\begin{enumerate}
\item[(1)]
There exists a disintegration of $\m\llcorner_{\mathcal{T}_{u}}$ on $\{X_{\alpha}\}_{\alpha\in{Q}}$:
\[
\m\llcorner_{\mathcal{T}_{u}}= \int_{Q} \m_{\alpha}\, \mathfrak{q}(d\alpha), \,\,\,\text{where}\, \m_{\alpha}(X_{\alpha})=1,\,\, \text{for}\,\mathfrak{q}\text{-a.e.} \,\alpha\in{Q}.
\]

\item[(2)] For $\mathfrak{q}$-a.e. $\alpha\in{Q}$, $X_{\alpha}$ is a \emph{transport ray} for $\Gamma_u$. 


\item[(3)] For $\mathfrak{q}$-a.e. $\alpha\in{Q}$, the metric measure space $(X_{\alpha},\sfd,\m_{\alpha})$ satisfies $\CD(K,N)$.
\end{enumerate}
We say that $(X,\sfd,\m)$ satisfies the $\CD^1(K,N)$ condition it is satisfies the $\CD^1_u(K,N)$ condition for every $u:X\to\R$ 1-Lipschitz.
\end{definition}

By transport ray, we mean that $X_{\alpha}$ is the image of a closed non-null geodesic $\gamma$ parametrized by arc length on an interval $I$ in such a way the function $u\circ\gamma$ is affine with slope $-1$ on $I$, moreover it is maximal with respect to inclusion.

\begin{remark}
\label{rem:cd1dim}
 It is well known  that 
the last condition of Definition \ref{D:defCD1} is equivalent to  
ask $\mm_{\alpha} \sim h_{\alpha} \mathcal{L}^{1}\llcorner_{[0,|X_{\alpha}|]}$ 
where $|X_{\alpha}|$ denotes the length of the transport ray $X_{\alpha}$ ($\sim$ means up to isometry of the space)  and the density $h_{\alpha}$ has to satisfy 
\begin{equation}
\label{eq:distrCD}
\left(h_{\alpha}^{1/(N-1)}\right)'' + \frac{K}{N-1}h_{\alpha}^{1/(N-1)}\leq 0,
\end{equation}
in the distributional sense. In turn this is equivalent to the fact that the continuous representative of $h_\alpha$ - which exists by \eqref{eq:distrCD} and that we shall continue to denote by $h_\alpha$ - satisfies 
\[
h_{\alpha}((1-t)R_0+tR_1)^{\frac{1}{N-1}}\geq \sigma_{K,N-1}^{(1-t)}(R_1-R_0)h_{\alpha}(R_0)^{\frac{1}{N-1}}+\sigma_{K,N-1}^{(t)}(R_1-R_0)h_{\alpha}(R_1)^{\frac{1}{N-1}},
\]
for any $R_0,R_1\in[0,|X_\alpha|]$, $R_0\leq R_1$, and $t\in[0,1]$.
\end{remark}

\section{Equivalent Formulations of Ricci Curvature bounds}\label{S:Main}

In this section we obtain the equivalence between $\CD^{1}(K,N)$ and $\CD_{1}(K,N)$.
Recall that to avoid pathologies we assume $\supp (\mm) = X$.

\begin{theorem}\label{T:Main}
Let $(X,\sfd,\m)$ be an essentially non-branching metric measure space
and further assume $\mm(X) = 1$. 
Then $(X,\sfd,\m)$ satisfies the $\CD^1(K,N)$ condition if and only if 
it satisfies the $\CD_{1}(K,N)$ condition.
\end{theorem}

We will present separately the two implications needed for the proof of Theorem \ref{T:Main}.

\subsection{$\CD_{1}(K,N) \implies \CD^1(K,N)$}
So consider fixed $u : X \to \R$ a $1$-Lipschitz function and $(X,\sfd,\mm)$ be 
essentially non-branching and verifying $\CD_{1}(K,N)$ with $\mm(X) = 1$.

\smallskip
   
{\bf Step 1.} Disintegration formula.\\
First notice that $\CD_{1}(K,N)$ implies, reasoning for instance like  
\cite{sturm:II} in the case $p =2$, that the space is proper.
Moreover $\CD_{1}(K,N)$ implies the following variant of $\MCP(K,N)$ 
(for the definition of $\MCP$ we refer to \cite{sturm:II} and \cite{Ohta1}): 
%
\begin{lemma}\label{le:32}
Let $(X,\sfd,\mm)$ be a m.m.s. with $\mm(X) = 1$ and satisfying $\CD_1(K,N)$. 
Then $(X,\sfd,\mm)$ satisfies the following version of $\MCP(K,N)$. 
For any $\mu_{0} \in \mathcal{P}(X)$ with $\mu_{0} \ll \mm$ and $x_{0}\in X$
there exists a curve $(\mu_t)$ which is a $W_{p}$-geodesic for any $p\in[1,\infty)$ 
such that $\mu_{t}= \rho_{t} \mm + \mu_{t}^{s}$ for all $t \in [0,1)$ then
\begin{equation}\label{E:MCP}
\int_{X} \rho_{t}^{-1/N'} \mu_{t} \geq \int_{X} \tau^{(1-t)}_{K,N'}(\sfd(x,x_{0}))  \rho_{0}^{-1/N'}(x) \mu_{0}(dx),
\end{equation}
for all $t \in [0,1)$ and $N'\geq N$.
\end{lemma}
\begin{proof}
Let $\mu_{0} \in \mathcal{P}(X,\sfd,\mm)$ and $x \in X$ be given. Since $\supp (\mm) = X$, 
we can consider $\mu_{1,\ve} : = c_{\ve}\mm\llcorner_{B_{\ve}(x_{1})}$, with 
$c_{\ve} > 0$ normalisation constant.
Let $\ppi_{\ve}$ be given by Definition \ref{D:CD1entropy} and put 
$\mu_t,\ve:=({\rm e}_t)_*\ppi_{\ve}$. It is classical to check that properness of $X$ 
implies that $\ppi_{\ve}$ is precompact and therefore we can obtain a limit 
dynamical plan $\ppi$ inducing a geodesic from $\mu_{0}$ to $\delta_{0}$. 
Validity of \eqref{E:MCP} simply follows by lower semicontinuity of entropy and the claim 
follows.
\end{proof}

The version of $\MCP(K,N)$ obtained in Lemma \ref{le:32} is actually equivalent to 
the classical one, provided the space is essentially non-branching: we refer for its 
proof to \cite[Lemma 6.13]{CMi} (see also \cite[Section 5]{R2012}). 
Hence in our framework we can directly use the classical $\MCP(K,N)$.

Immediately we deduce that $X$ is a geodesic space. Hence, as discussed in Section \ref{Ss:1OTu}, the following disintegration formula is valid: 
\begin{equation}\label{E:disint2}
\m \llcorner_{\T_{u}} = \m \llcorner_{\T^b_{u}}= \int_{Q} \m_{\alpha} \, \mathfrak{q}(d\alpha),
\end{equation}
where  $\qq = f_{\sharp} (\m \llcorner_{\T_{u}^b})$, and for 
$\qq$-a.e. $\alpha \in Q$
we have $\mm_{\alpha} \in \mathcal{P}(X)$, 
with $\mm_{\alpha}(X \setminus X_{\alpha}) = 0$: the notation $X_{\alpha}$ is used to denote the equivalence class of the element $\alpha \in Q$ that is, in particular, a transport ray.
Notice that  the first identity follows from the essentially non-branching assumption and
the discussion after Theorem \ref{teo:branch}.

Hence it is only left to show that $(X_{\alpha},\sfd,\mm_{\alpha})$ satisfy $\CD(K,N)$.

\smallskip
{\bf Step 2.} Intermediate regularity of conditional measures.\\
It is already present in the literature how to improve the validity of \eqref{E:MCP} 
to any $\mu_{1} \in \mathcal{P}(X)$, provided the space is essentially non-branching and the 
geodesic $(\mu_{t})_{t\in [0,1]}$ is  a $W_{2}$-geodesic. 

This will be enough to deduce a first result on the regularity of $\mm_{\alpha}$.
Indeed localization for $\MCP(K,N)$ was, in a different form, already known 
in 2009, see \cite[Theorem 9.5]{biacava:streconv}, for non-branching m.m.s.. 
The case of essentially non-branching m.m.s.'s and an effective reformulation (after the work of Klartag \cite{klartag}) have been recently discussed in \cite[Section 3]{CM18} to which we refer for all the missing details (see in particular \cite[Theorem 3.5]{CM18}). 
Here we briefly report the following fact:

\smallskip
If $(X,\sfd,\mm)$ is an essentially non-branching m.m.s. with $\supp(\mm) = X$ and 
satisfying $\MCP(K,N)$, for some $K\in \R, N\in (1,\infty)$,  
then, for $\qq$-a.e.$\alpha$, $\mm_{\alpha}=h_{\alpha} \cH^{1}\llcorner_{X_{\alpha}}$ and 
the one-dimensional metric measure space $(X_{\alpha},\sfd,\m_{\alpha})$  
verifies $\MCP(K,N)$; in particular 
$h_{\alpha}$ is strictly positive in the relative interior of $X_{\alpha}$ and locally Lipschitz.

\smallskip
{\bf Step 3.} $\CD(K,N)$ estimates for one-dimensional spaces.\\
In order to conclude, it remains to show that  for $\qq$-a.e. $\alpha\in{Q}$, the one-dimensional metric measure space $(X_{\alpha},\sfd,\m_{\alpha})$ satisfies $\CD(K,N)$. 
It is useful to introduce the following  \emph{ray map} 
$g: \dom(g)\subset Q\times \mathbb{R} \to \T_{u}$, defined as follows:
\begin{equation*}
\begin{split}
 \gra(g):=&\big\{(\alpha,t,x)\in{Q\times[0,+\infty)}\times \T_{u}^{b}: (\alpha,x)\in{\Gamma}, \sfd(\alpha,x)=t
 \big\}
 \\ 
 &\cup \, \big\{(\alpha, t,x)\in{Q\times (-\infty,0]\times\T_{u}^{b}}:(x,\alpha)\in{\Gamma}, 
 \sfd(x,\alpha)=t\big\}.
\end{split}
\end{equation*}
\noindent
The ray map $g$ enjoys several  properties 
already obtained in \cite[Proposition 5.4]{cava:MongeRCD}:
\begin{itemize}
\item $g$ is a Borel map;
\item $t\mapsto g(\alpha,t)$ is an isometry. If $s,t\in{Dom(\alpha, \cdot)}$ with $s\leq t$, then $(g(\alpha,s),g( \alpha,t))\in{\Gamma}$;
\item $\dom(g)\ni (\alpha, t)\mapsto g(\alpha,t)$ is bijective on $f^{-1}(Q)\subset \T_{u}^b$.
\end{itemize}
In particular, via $g$ we will identify the set of definition of the densities $h_{\alpha}$ 
with real intervals.

We start with  the following preliminary result.

\begin{lemma}
For any $\bar{Q}\subseteq Q$  Borel set with positive $\qq$-measure and  for  $R_0,R_1,L_0,L_1\in{\mathbb{R}}$ such that $R_0<R_1$, $L_0,L_1>0$ and
$[R_{0}, R_{1} + L_{1}]$ belongs to the domain of $\qq$-a.e. $h_{\alpha}$,  it holds:
\begin{align}\label{equ:firstclaim}
& (L_t)^{\frac{1}{N}}\sup_{\bar{Q}}h_{\alpha}^{\frac{1}{N}}(R_t)\nonumber \\
& \geq (L_0)^{\frac{1}{N}} \tau_{K,N}^{(1-t)}(\sfd(R_0,R_1))\inf_{\bar{Q}}h_{\alpha}^{\frac{1}{N}}(R_0)+(L_1)^{\frac{1}{N}} \tau_{K,N}^{(t)}(\sfd(R_0,R_1))\inf_{\bar{Q}}h_{\alpha}^{\frac{1}{N}}(R_1),
\end{align}
for every $t\in{[0,1]}$, where $R_t=(1-t)R_0+tR_1$ (the same  holds for $L_t$).
\end{lemma}

\begin{proof}
{\bf Step 1.}\\
Fix $\bar{Q}\subseteq Q$  Borel set with positive $\qq$-measure and consider $R_0,R_1,L_0,L_1\in{\mathbb{R}}$ such that $R_0<R_1$ and $L_0,L_1>0$. Define for $i=1,2$ the probability measures:
\[
\mu_i=\frac{1}{\qq(\bar{Q})} \int_{\bar{Q}} g(\alpha,\cdot)_{\sharp}\biggl{(}\frac{1}{\varepsilon L_i} \mathcal{L}^1\llcorner_{[R_i, R_i+\varepsilon L_i]} \biggr) \qq(d\alpha).
\]
First of all  observe that,  for such measures, the transport has to be performed along  the rays $\{X_\alpha\}_{\alpha\in{\bar{Q}}}$. For sure  an optimal  plan with this property exists, since the plan $\pi$  rearranging  the mass monotonically along each ray is optimal; hence $\supp \, \pi \subset \Gamma$, so it 
is $\sfd$-cyclically monotone and therefore $W_{1}$-optimal. 
The aim is to prove that all the other optimal plans  enjoy the  same property.

Indeed,  if not, there would exist at least one optimal plan $\bar{\pi}$ such that, for some $\bar{Q}_1 \subset \bar{Q}$ of positive $\qq$-measure and for some $S\subset{\mathbb{R}}$, it holds
\[
\bar{\pi}\{(g(\alpha,s),g(\alpha',s')) :\alpha,\alpha'\in{\bar{Q}_1},  s,s'\in{S} \,\text{with}\,  \alpha\neq \alpha'\}>0,
\]
with $\bar Q_{1} \times S \subset \dom(g)$.
Let  us consider the plan
\[
\pi^*=\frac{\pi+\bar{\pi}}{2};
\]
trivially, it is still optimal for the couple $\mu_0,\mu_1$.  By construction this plan splits some  points, generating in this way  a set of branching points with positive measure. 
This will lead to a contradiction. Consider indeed the Kantorovich potential $v$ associated to the 
$W_{1}$-optimal transport problem between $\mu_{0}$ and $\mu_{1}$, possibly different from the 1-Lipschitz function $u$ we fixed above. Theorem \ref{teo:branch}
applied to $v$ implies 
that necessarily that $\mm(A_{\pm,v}) = 0$.    
Since $A_{\pm,v}$ will contain $P_{1}(\{(g(\alpha,s),g(\alpha',s')) :\alpha,\alpha'\in{\bar{Q}_1},  s,s'\in{S} \,\text{with}\,  \alpha\neq \alpha'\})$ considered above, and $\mu_{0} \ll \mm$, the contradiction
with $\bar\pi (\{(g(\alpha,s),g(\alpha',s')) :\alpha,\alpha'\in{\bar{Q}_1},  s,s'\in{S} \,\text{with}\,  \alpha\neq \alpha'\}) > 0$ follows.
Hence, every optimal plan will have support contained in the set
$$
 A_{\bar{Q}}^{\varepsilon}:=\cup_{\alpha\in{\bar{Q}}} g(\alpha, [R_0,R_0+\varepsilon L_0])\times g(\alpha, [R_1,R_1+\varepsilon L_1]).
$$

{\bf Step 2.}\\
Since by definition $\mu_0,\mu_1\ll \m$,  there exists a dynamic transport plan $\ppi$ as in Definition \ref{D:CD1entropy} such that for $\mu_t:=({\rm e}_t)_*\ppi=\rho_t\mm$
the inequality  \eqref{equ:cd} holds true. Step 1 above and the fact that $\ppi$ is concentrated on constant speed geodesics completely characterize $\ppi$; in particular we have that for $q$-a.e.\ $\alpha\in Q$ the function $\rho_t$ is 0 $m_\alpha$-a.e.\ outside the `interval' $g(\alpha,[R_t,R_t+\eps L_t])$. Hence using the Disintegration Theorem and   Jensen inequality we can estimate the left-hand side of \eqref{equ:cd} by:
\begin{align*}
\int_{X} \rho_t^{1-\frac{1}{N}}\,d\m &= \int_{\bar{Q}} \int_{X_{\alpha}} \rho_t(x)^{1-\frac{1}{N}} m_{\alpha}(dx) \qq(d\alpha)=\int_{\bar{Q}} \int_{R_t}^{R_t+\varepsilon L_t} \rho_t(g(\alpha, s))^{1-\frac{1}{N}}h_{\alpha}(s)ds\qq(d\alpha)\\
& \leq (\varepsilon L_t) \int_{\bar{Q}}\sup_{[R_t,R_t+\varepsilon L_t]}h_{\alpha}^{\frac{1}{N}} \fint_{R_t}^{R_t+\varepsilon L_t} \bigr(\rho_t(g(\alpha, s))h_{\alpha}(s)\bigr)^{1-\frac{1}{N}} ds\qq(d\alpha)\\
&\leq (\varepsilon L_t)^{\frac{1}{N}}\int_{\bar{Q}}\sup_{[R_t,R_t+\varepsilon L_t]}h_{\alpha}^{\frac{1}{N}} \biggl(\int_{R_t}^{R_t+\varepsilon L_t} \rho_t(g(\alpha, s))h_{\alpha}(s) ds\biggr)^{1-\frac{1}{N}}\qq(d\alpha)\\
&\leq (\varepsilon L_t\qq(\bar{Q}))^{\frac{1}{N}} \sup_{\bar{Q}}\biggl(\sup_{[R_t,R_t+\varepsilon L_t]}h_{\alpha}^{\frac{1}{N}}\biggr).
\end{align*}
Arguing similarly, the right-hand side of \eqref{equ:cd} can be estimated in the following way where $\pi = (\ee_{0},\ee_{1})_{\sharp} \ppi$:
\begin{align*}
& \int_{X\times X} \rho_0^{-\frac{1}{N}}(x) \tau_{K,N}^{(1-t)}(\sfd(x,y))+ \rho_1^{-\frac{1}{N}}(y) \tau_{K,N}^{(t)}(\sfd(x,y))\pi(dx,dy)  \\
&~ \geq 
\inf_{ A_{\bar{Q}}^{\varepsilon}} \tau_{K,N}^{(1-t)}(\sfd(x,y))\int_{X}\rho_0^{1-\frac{1}{N}}(x)\m(dx)+
\inf_{ A_{\bar{Q}}^{\varepsilon}} \tau_{K,N}^{(t)}(\sfd(x,y))\int_{X}\rho_1^{1-\frac{1}{N}}(y)\m(dy)\\
&~\geq (\varepsilon \qq(\bar{Q}))^{\frac{1}{N}} \biggl[\inf_{ A_{\bar{Q}}^{\varepsilon}} \tau_{K,N}^{(1-t)}(\sfd(x,y)) \inf_{\bar{Q}}\biggl(\inf_{[R_0,R_0+\varepsilon L_0]} h_{\alpha}^{\frac{1}{N}}\biggr)(L_0)^{\frac{1}{N}} \\
&~\qquad
+\inf_{ A_{\bar{Q}}^{\varepsilon}} \tau_{K,N}^{(t)}(\sfd(x,y)) \inf_{\bar{Q}}\biggl(\inf_{[R_1,R_1+\varepsilon L_1]} h_{\alpha}^{\frac{1}{N}}\biggr)(L_1)^{\frac{1}{N}} \biggr].
\end{align*}
Hence, considering both the estimates obtained so far, we get 
\begin{align*}
(L_t)^{\frac{1}{N}}\sup_{\bar{Q}}\biggl(\sup_{[R_t,R_t+\varepsilon L_t]}h_{\alpha}^{\frac{1}{N}}\biggr) 
\geq & ~  \inf_{ A_{\bar{Q}}^{\varepsilon}} \tau_{K,N}^{(1-t)}(\sfd(x,y)) \inf_{\bar{Q}}\biggl(\inf_{[R_0,R_0+\varepsilon L_0]} h_{\alpha}^{\frac{1}{N}}\biggr)(L_0)^{\frac{1}{N}} \\
& ~ +\inf_{ A_{\bar{Q}}^{\varepsilon}} \tau_{K,N}^{(t)}(\sfd(x,y))\inf_{\bar{Q}}\biggl(\inf_{[R_1,R_1+\varepsilon L_1]} h_{\alpha}^{\frac{1}{N}}\biggr)(L_1)^{\frac{1}{N}}.
\end{align*}
Sending $\varepsilon\to 0$, we obtain 
\begin{equation*}
(L_t)^{\frac{1}{N}}\sup_{\bar{Q}}h_{\alpha}^{\frac{1}{N}}(R_t)\geq (L_0)^{\frac{1}{N}}\inf_ {A_{\bar{Q}}} \tau_{K,N}^{(1-t)}(\sfd(x,y))\inf_{\bar{Q}}h_{\alpha}^{\frac{1}{N}}(R_0)+(L_1)^{\frac{1}{N}}\inf_{ A_{\bar{Q}}} \tau_{K,N}^{(t)}(\sfd(x,y))\inf_{\bar{Q}}h_{\alpha}^{\frac{1}{N}}(R_1),
\end{equation*}
where $A_{\bar{Q}}:=\cup_{\alpha\in{\bar{Q}}} \{(g(\alpha, R_0), g(\alpha, R_1))\}$. Since $g(\alpha,\cdot)$ is an isometry, \eqref{equ:firstclaim} is proved.
\end{proof}

\bigskip
We are now ready to prove the following:
\begin{proposition}
 For $\qq$-a.e. $\alpha\in{Q}$, the metric measure space $(X_{\alpha},\sfd,\m_{\alpha})$ satisfies $\CD(K,N)$.
\end{proposition}
\begin{proof}
By remark \ref{rem:cd1dim},  to prove the claim   is sufficient to show that:
\begin{equation}
\label{equ:claimcd}
h_{\alpha}((1-t)R_0+tR_1)^{\frac{1}{N-1}}\geq \sigma_{K,N-1}^{(1-t)}(R_1-R_0)h_{\alpha}(R_0)^{\frac{1}{N-1}}+\sigma_{K,N-1}^{(t)}(R_1-R_0)h_{\alpha}(R_1)^{\frac{1}{N-1}},
\end{equation}
for all $t\in{[0,1]}$ and for $R_0,R_1\in{ [0,L_{\alpha}]}$ with $R_0<R_1$,
where we have identified the transport ray $X_{\alpha}$ 
with the real interval $[0,L_{\alpha}]$ having the same length.

\noindent
As already did in \cite{CM17a}, it is sufficient to show that for every  $R_0,R_1\in{ [0,L_{\alpha}]}$ with $R_0<R_1$ and $L_0,L_1>0$, we have that for  $\qq$-a.e. $\alpha\in{Q}$ 
\begin{equation}
\label{equ:cmclaim}
(L_t)^{\frac{1}{N}}h_{\alpha}^{\frac{1}{N}}(R_t)\geq(L_0)^{\frac{1}{N}} \tau_{K,N}^{(1-t)}(\sfd(R_0,R_1))h_{\alpha}^{\frac{1}{N}}(R_0)+(L_1)^{\frac{1}{N}} \tau_{K,N}^{(t)}(\sfd(R_0,R_1))h_{\alpha}^{\frac{1}{N}}(R_1),
\end{equation}
for all $t\in{[0,1]}$, where $L_t=(1-t)L_0+tL_1$ (the same for $R_t$). Indeed, if this is the case taking also into account the already established continuity of $h_\alpha$, one can make the choice
\[
L_0=\frac{\sigma_{K,N-1}^{(1-t)}(\sfd(R_0,R_1)) h(R_0)^{\frac{1}{N-1}}}{1-t}, \quad 
L_1=\frac{\sigma_{K,N-1}^{(t)}(\sfd(R_0,R_1)) h(R_1)^{\frac{1}{N-1}}}{t},
\]
obtaining exactly \eqref{equ:claimcd}. Thus, our aim will be proving \eqref{equ:cmclaim}.
Arguing by contraddiction, let us assume that there exist $R_0,R_1\in{[0, L_{\alpha}]}$,
$L_{0},L_{1}>0$ with $R_{0} + L_{0}, R_{1} + L_{1} < L_{\alpha}$
and a Borel set $Q_1\subseteq Q$  with positive $\qq$-measure such that for every $\alpha\in{Q_1}$ it holds:
\begin{equation}
(L_t)^{\frac{1}{N}}h_{\alpha}^{\frac{1}{N}}(R_t)<(L_0)^{\frac{1}{N}} \tau_{K,N}^{(1-t)}(\sfd(R_0,R_1))h_{\alpha}^{\frac{1}{N}}(R_0)+(L_1)^{\frac{1}{N}} \tau_{K,N}^{(t)}(\sfd(R_0,R_1))h_{\alpha}^{\frac{1}{N}}(R_1).
\end{equation}
By Lusin Theorem, there exists a Borel set $Q_2\subset Q_1$ with positive $\qq$-measure on which the maps $\alpha \mapsto h_{\alpha}(R_i)$, for $i=0,t,1$ are continuous. Hence, fixed $\delta>0$, there exists $Q_3 \subset Q_2$ with positive $\qq$-measure  such that
\[
(L_t)^{\frac{1}{N}}h_{\alpha}^{\frac{1}{N}}(R_t)<(L_0)^{\frac{1}{N}}\tau_{K,N}^{(1-t)}(\sfd(R_0,R_1))h_{\alpha}^{\frac{1}{N}}(R_0)+(L_1)^{\frac{1}{N}} \tau_{K,N}^{(t)}(\sfd(R_0,R_1))h_{\alpha}^{\frac{1}{N}}(R_1)-\delta,\,\,\,\, \forall \alpha\in{Q_3}.
\]
In particular, for every $\bar{Q}\subset Q_3$ compact set  with positive $\qq$-measure:
\[
(L_t)^{\frac{1}{N}}\sup_{\bar{Q}} h_{\alpha}^{\frac{1}{N}}(R_t)<(L_0)^{\frac{1}{N}}\tau_{K,N}^{(1-t)}(\sfd(R_0,R_1))\sup_{\bar{Q}} h_{\alpha}^{\frac{1}{N}}(R_0)+(L_1)^{\frac{1}{N}}\tau_{K,N}^{(t)}(\sfd(R_0,R_1))\sup_{\bar{Q}} h_{\alpha}^{\frac{1}{N}}(R_1)-\delta.
\]
Combining the latter inequality with \eqref{equ:firstclaim}, we deduce that for any $\bar{Q}\subset Q_3$ Borel set  with positive $\qq$-measure
\begin{align*}
&(L_0)^{\frac{1}{N}}\tau_{K,N}^{(1-t)}(\sfd(R_0,R_1))\inf_{\bar{Q}}h_{\alpha}^{\frac{1}{N}}(R_0)+(L_1)^{\frac{1}{N}}\tau_{K,N}^{(t)}(\sfd(R_0,R_1))\inf_{\bar{Q}}h_{\alpha}^{\frac{1}{N}}(R_1) <\\
& (L_0)^{\frac{1}{N}}\tau_{K,N}^{(1-t)}(\sfd(R_0,R_1))\sup_{\bar{Q}} h_{\alpha}^{\frac{1}{N}}(R_0)+(L_1)^{\frac{1}{N}}\tau_{K,N}^{(1-t)}(\sfd(R_0,R_1)) \sup_{\bar{Q}} h_{\alpha}^{\frac{1}{N}}(R_1)-\delta.
\end{align*}
Since the parameter $\delta$ does not depend on $\bar{Q}$, we obtain a contradiction.
\end{proof}

This concludes the proof of the implication: from $\CD_1(K,N)$ to $\CD^{1}(K,N)$.
We will next move to the opposite implication.

\subsection{$\CD^{1}(K,N) \implies \CD_{1}(K,N)$}

Notice that $\CD^{1}(K,N)$ implies that $(X,\sfd,\mm)$ is a proper geodesic space 
and verifies $\MCP(K,N)$ (see for all the details \cite{CMi}).

Let $\mu_0, \mu_1\in{\mathcal{P}_1(X,\sfd,\m)}$  be given. 
We will construct a $W_{1}$-geodesic verifying the Entropy inequality. 
Consider therefore $u:X\to \mathbb{R}$ a  Kantorovich potential associated to the 
transport problem between $\mu_0,\mu_1$ with cost $\sfd$.
Consider the associated $\Gamma_{u}$; then any optimal transport plan 
$\pi$ has to be concentrated over $\Gamma_{u}$, i.e. $\pi(\Gamma_{u})=1$. 
Moreover, with no loss in generality we can assume that $\mu_{0}$ is concentrated over the 
transport set $\T_{u}^{b}$: indeed the part of $\mu_{0}$ outside of $\T_{u}^{b}$ 
is left in place by $\pi$; in particular, it will not give any contribution in the 
Entropy inequality as $\tau_{K,N}^{(1-t)}(0) = 0$.

Since $u$ is $1$-Lipschitz, by the  $\CD^1_u(K,N)$ condition there exist a family 
of rays $\{X_{\alpha}\}_{\alpha\in{Q}}\subset X$ and a disintegration 
of $\m\llcorner_{\mathcal{T}_{u}}$ on $\{X_{\alpha}\}_{\alpha\in{Q}}$ such that:
\begin{equation}\label{E:Disintm}
\m\llcorner_{\mathcal{T}_{u}}=  \m\llcorner_{\T_{u}^{b}} = \int_{Q} \m_{\alpha}\, \mathfrak{q}(d\alpha), \,\,\,\text{where}\, \m_{\alpha}(X_{\alpha})=1,\,\, \text{for}\,\mathfrak{q}\text{-a.e.} \,\alpha\in{Q},
\end{equation}
where the first identity is given by Theorem \ref{teo:branch} 
 and $(X_{\alpha},\sfd, \m_{\alpha})\in{\CD(K,N)}$.  It follows that 
\begin{equation}
\mu_0 =\rho_0\m= \int_{Q} \rho_{0}\,\m_{\alpha}\,\qq(d\alpha )=\int_{Q} \mu_{0,\alpha}\qq_0(d\alpha)
\end{equation}
where  $\mu_{0,\alpha}=\rho_{0}\m_{\alpha}\cdot(\int \rho_0\m_{\alpha})^{-1}$ and 
$\q_0=f_{\sharp}(\mu_0)$ with $f$   the quotient map. 
Then we claim that  for any Borel  set $C\subseteq Q$  it holds:
\[
(f^{-1}(C)\times X)\cap (\Gamma_{u}\setminus\{x=y\})\cap (\mathcal{T}_{u}^{b}\times \mathcal{T}_{u}^{b})=
( X\times f^{-1}(C))\cap (\Gamma_{u} \setminus\{x=y\})\cap (\mathcal{T}_{u}^{b}\times \mathcal{T}_{u}^{b}).
\]
Indeed, since $\mu_0(\T_{u}^{b})=\mu_1(\T_{u}^{b})=1$, 
then $\pi((\Gamma_{u}\setminus\{x=y\})\cap \mathcal{T}^{b}_{u}\times \mathcal{T}_{u}^{b})=1$; hence 
if $x,y\in{\T_{u}^{b}}$ with $(x,y)\in{\Gamma_{u}}$, then it must be $f(x)=f(y)$ since $\T_{u}^{b}$ 
does not admit forward or backward branching points. This implies that  
\begin{align*}
\mu_0(f^{-1}(C))&= \pi((f^{-1}(C)\times X)\cap (\Gamma_{u}\setminus\{x=y\}))\\
&= \pi ( X\times f^{-1}(C))\cap (\Gamma_{u} \setminus\{x=y\})\\
&= \mu_1(f^{-1}(C));
\end{align*}
in particular $\qq_{0} = \qq_{1} : = f_{\sharp}(\mu_{1})$.
Hence, we can write the following disintegration:
$
\mu_1=\rho_1\m= \int_{Q} \rho_{1}\,\m_{\alpha}\,\qq(d\alpha )=\int_{Q} \mu_{1,\alpha}\qq_0(d\alpha),
$
where $\mu_{1,\alpha}=\rho_{1}\m_{\alpha}\cdot(\int \rho_0\m_{\alpha})^{-1}$ 
and, $\qq_{0}$-a.e., $\mu_{0,\alpha},\mu_{1,\alpha}$ are probability measures on $X_{\alpha}$. 
Furthermore, by construction they are 
absolutely continuous with respect to  $\m_{\alpha}$. By the $\CD^1_u(K,N)$ condition,  the metric measure space $(X_{\alpha},\sfd,\m_{\alpha})$ satisfies $\CD(K,N)$ and hence there exists an optimal dynamical plan $\nu_{\alpha}$ such  that $\rho_{t,\alpha}\mm_{\alpha} = \mu_{t,\alpha}=(\ee_t)_{\sharp}\nu_{\alpha} $ is a $W^1$-geodesic interpolating $\mu_{0,\alpha}$ ad $\mu_{1,\alpha}$   and
\begin{equation}
\label{equ:cdpunt}
\rho_{t,\alpha}^{-\frac{1}{N'}}(\gamma_t)\geq \tau_{K,N'}^{(1-t)}(\sfd(\gamma_0,\gamma_1)) \rho_{0,\alpha}^{-\frac{1}{N'}}(\gamma_0)+\tau_{K,N}^{(t)}(\sfd(\gamma_0,\gamma_1))\rho_{1,\alpha}^{-\frac{1}{N'}}(\gamma_1), \,\,\, \text{for } \nu_{\alpha}  \text{ a.e.\,} \gamma.
\end{equation}
It is then natural to proceed gluing 1-dimensional geodesics: define
$ \nu=\int_{Q} \nu_{\alpha}\qq_0(d\alpha)$  and set $\mu_t=(\ee_t)_{\sharp}\nu$. 
Observe that, it holds
$\mu_t= \int_{Q} \mu_{t,\alpha}\qq_0(d\alpha)$ and 
we claim that $\{\mu_t\}$ is a $W_1$-geodesic interpolating $\mu_0$ and $\mu_1$. 
Indeed:
\begin{align*}
W_{1}(\mu_t,\mu_s) &\leq \int_{X\times X} \sfd(x,y) (\ee_t,\ee_s)_{\sharp}\nu(dxdy)\\
&=\int_{Q}\int_{X_{\alpha}\times X_{\alpha}}  \sfd(x,y) (\ee_t,\ee_s)_{\sharp}\nu_{\alpha}(dxdy) \qq_0(d\alpha) \\
&= |t-s| \int_{Q}\int_{X_{\alpha}\times X_{\alpha}}  \sfd(x,y) (\ee_0,\ee_1)_{\sharp}\nu_{\alpha}(dxdy) \qq_0(d\alpha) \\
&=|t-s| \int_{X\times X}\sfd(x,y) (\ee_0,\ee_1)_{\sharp}\nu(dxdy)\\
&=|t-s|W_1(\mu_0,\mu_1).
\end{align*}
The last equality  follows from the optimality of the plan: indeed  $(\ee_0,\ee_1)_{\sharp}\nu$
 is concentrated on a $\sfd$-cyclically monotone with  marginals  $\mu_0$ and $\mu_1$.
To conclude, we show  the convexity inequality  \eqref{equ:cd} along the geodesic $\mu_t$. 

If $\mu_{t} = \rho_{t} \mm$, it follows from \eqref{E:Disintm} that for 
each $t\in{[0,1]}$ it holds $\rho_{t,\alpha}= \frac{\rho_t}{\int \rho_0 m_{\alpha}}$.
Hence the inequality \eqref{equ:cdpunt} can be rewritten in the following way:
\begin{equation}
\label{equ:cdespl}
\rho_{t}^{-\frac{1}{N'}}(\gamma_t)\geq \tau_{K,N'}^{(1-t)}(\sfd(\gamma_0,\gamma_1)) \rho_{0}^{-\frac{1}{N'}}(\gamma_0)+\tau_{K,N}^{(t)}(\sfd(\gamma_0,\gamma_1))\rho_{1}^{-\frac{1}{N'}}(\gamma_1), \,\,\, \text{for } \nu_{\alpha}\text{-a.e.\,} \gamma.
\end{equation}
Since for $\qq_0$-a.e. $\alpha$ the inequality \eqref{equ:cdespl} holds  for $\nu_{\alpha}-$a.e. $\gamma$,  a fortiori  it holds true for $\nu-$a.e. $\gamma$; hence, the claim is proved.

\end{document}